\documentclass{article}

\title{The Nakayama automorphism of a\\self-injective preprojective algebra}
\author{Joseph Grant\\University of East Anglia, Norwich, UK\\\texttt{j.grant@uea.ac.uk}}
\date{\vspace{-2em}}

\usepackage{latexsym}
\usepackage{verbatim}
\usepackage{amsmath}
\usepackage{amssymb}
\usepackage{mathrsfs}
\usepackage{amsthm}
\usepackage{sfmath}
\usepackage{stmaryrd} 
\usepackage{graphicx} 
\usepackage{hyperref}  

\usepackage[british]{babel}

\input xy
\xyoption{all}
\newcommand{\F}{\mathbb{F}}
\newcommand{\Hom}{\operatorname{Hom}\nolimits}
\newcommand{\arr}[1]{\stackrel{#1}{\to}}
\newcommand{\mMod}{\operatorname{-mod}}
\newcommand{\Modm}{\operatorname{mod-}}

\newcommand{\op}{{\operatorname{op}\nolimits}}
\newcommand{\add}{\operatorname{-add}}

\newcommand{\da}{\text{-}}
\newcommand{\Ext}{\operatorname{Ext}\nolimits}

\newcommand{\stmod}{\operatorname{-\underline{mod}}}
\newcommand{\costmod}{\operatorname{-\overline{mod}}}

\newcommand{\larr}[1]{\stackrel{#1}{\longrightarrow}}
\newcommand{\Z}{\mathbb{Z}}

\newcommand{\grsh}[1]{\{{#1}\}}

\newcommand{\gen}[1]{\langle#1\rangle}

\newcommand{\id}{\operatorname{id}\nolimits}

\newcommand{\iP}{\mathbin{\rotatebox[origin=c]{180}{$\Pi$}}}

\newtheorem{theorem}{Theorem}[section]

\newtheorem{corollary}[theorem]{Corollary}

\newtheorem{proposition}[theorem]{Proposition}
\newtheorem{definition-proposition}[theorem]{Definition-Proposition}

\newtheorem*{thma}{Theorem A}
\newtheorem*{thmb}{Theorem B}

\theoremstyle{definition}
\newtheorem{definition}[theorem]{Definition}
\newtheorem{remark}[theorem]{Remark}
\newtheorem{example}[theorem]{Example}

\begin{document}

\maketitle

\begin{abstract}
We give a simple proof, using Auslander-Reiten theory, that the preprojective algebra of a basic hereditary algebra of finite representation type is Frobenius.  We then describe its Nakayama automorphism, which is induced by the Nakayama functor on the module category of our hereditary algebra.

\end{abstract}

\tableofcontents

\setlength{\parindent}{0pt}
\setlength{\parskip}{1em plus 0.5ex minus 0.2ex}

\section{Introduction}

\subsection{History}
Preprojective algebras were introduced by Gabriel and Ponomarev as a way to combine all the preprojective representations of a quiver into a single algebra \cite{gp}.  The path algebra of the quiver, or its opposite, is a subalgebra of the preprojective algebra, and on restriction to this subalgebra we obtain all the indecomposable preprojective modules up to isomorphism.  Preprojective algebras were originally defined by an explicit presentation: one doubles the starting quiver and quotients by an admissible ideal.  Later, Baer, Geigle, and Lenzing found a more conceptual definition which works for any hereditary algebra \cite{bgl}.  Their definition uses the natural multiplication on morphisms between preprojective modules.  Ringel and Crawley-Boevey proved that, if one gets the signs right, these definitions give isomorphic algebras \cite{rin,cb-quiv}.

Gabriel showed that the representation type of a quiver depends on the underlying graph: a quiver is representation finite precisely when it is an orientation of an ADE Dynkin diagram \cite{gab-ade}.  So one expects the associated preprojective algebras to have different properties in the Dynkin and non-Dynkin cases.  If our quiver has finitely many isoclasses of indecomposable modules then in particular it has finitely many isoclasses of indecomposable preprojective modules, thus its preprojective algebra is finite-dimensional.  If instead we start with a quiver of infinite representation type then one can find infinitely many non-isomorphic indecomposable preprojective modules, and thus the preprojective algebra is infinite-dimensional.

In the Dynkin case, the preprojective algebra is not only finite-dimensional: it is self-injective.  This is a much stronger property than just being finite-dimensional.  It says that every projective module is injective, and vice-versa.  As indecomposable projectives are projective covers of simple modules, and indecomposable injective modules are injective envelopes of simple modules, we obtain a permutation of the simple modules for our algebra, known as the Nakayama permutation.

Self-injectivity is a Morita-invariant property which minimally encompasses Frobenius algebras: an algebra is self-injective if and only if it is Morita equivalent to a Frobenius algebra.  An algebra is Frobenius when its left regular module is isomorphic to the dual of its right regular module.  Frobenius algebras come with a distinguished outer automorphism known as the Nakayama automorphism which induces the Nakayama permutation.

The fact that preprojective algebras of Dynkin quivers are self-injective seems to have been a folklore result for some time.  The first written statement known to the author is by Ringel and Schofield in their handwritten manuscript \cite{rs}.  They give an automorphism of each ADE Dynkin graph and state that this induces the Nakayama permutation of the preprojective algebra.  There is related work by Auslander and Reiten \cite{ar-dtr} and Buchweitz \cite{b-fin}.

Later, Brenner, Butler, and King made a detailed study of the ADE preprojective algebras as part of their study of trivial extension algebras with periodic bimodule resolutions.  They gave an explicit formula for the Nakayama antomorphism and wrote out detailed checks that it satisfies the Frobenius algbera condition in each of the ADE cases, thereby showing that these preprojective algebras are self-injective \cite{bbk}.  These checks rely on combinatorics specific to each Dynkin diagram and are quite involved, especially for types D and E.

Up to this point, all proofs of self-injectivity relied on Gabriel's classification theorem and then used combinatorics of Dynkin diagrams.  One might hope for a direct proof that the preprojective algebra of a representation finite hereditary algebra is self-injective.  This was achieved by Iyama and Oppermann as a special case of a much more general result, as part of their study of stable categories of higher preprojective algebras \cite{io-stab}.  Their proof uses much more sophisticated technology: they deduce self-injectivity of the preprojective algebra from the stability of a cluster-tilting subcategory of the derived category under the Serre functor.  Since Iyama and Oppermann's result, other proofs of self-injectivity have been found, such as \cite[Corollary 12.7]{gls} and \cite[Corollary 4.13]{gi}.

\subsection{Results}
This article contains two main results.  The first is a new proof of the result we have been discussing:

\begin{thma}[{Theorem \ref{thm:pi-is-frob}}]
The preprojective algbera of a representation-finite hereditary algebra is self-injective.
\end{thma}

Like the proof of Iyama and Oppermann, this proof does not rely on Gabriel's classification theorem.  It uses relatively basic technology, avoiding the derived category.  The preprojective algebra is defined using only the module category, and we see it as a virtue that the proof stays within this setting.  We assume knowledge of adjunctions \cite[Chapter 4]{macl} and some Auslander-Reiten theory of hereditary algebras \cite{ars}

Our proof has two key ingredients.  The first is Serre duality: roughly, maps out of a projective module are dual to maps into the corresponding injective module.
The second is a result proved by both Platzeck-Auslander and Gabriel which characterizes representation finite hereditary algebras in terms of whether their injective modules are preprojective.  This result was also crucial in the proofs of self-injectivity by Brenner-Butler-King and Iyama-Oppermann.

Brenner-Butler-King described the Nakayama automorphism of an ADE preprojective algebra using its Gelfand-Ponomarev presentation by quivers with relations.  One might wonder if this automorphism has any interpretation in terms of the representation theory of the original quiver: can we describe the Nakayama automorphism for the Baer-Geigle-Lenzing preprojective algebra?  Our second main result does this.

\begin{thmb}[{Theorem \ref{thm:nak}}]
The Nakayama automorphism of the preprojective algebra is induced by the Nakayama functor of the representation-finite hereditary algebra.
\end{thmb}

It may be worth clarifying that, for any self-injective algebra, the Nakayama automorphism induces the Nakayama functor.  The above theorem is saying something different: the Nakayama functor for the hereditary algebra induces the Nakayama automorphism for the preprojective algebra.

This result appears to be new, though I suspect that it will not surprise the experts.  
Once one realizes that the Frobenius isomorphism relies on Serre duality, it makes sense that the right action should be twisted by the Nakayama functor.

Given this intuition, the phrase ``is induced by'' needs to be made precise.  The problem is that the Nakayama functor doesn't behave well on the module category: as our algebra is hereditary it kills all non-projective modules.  There are two possible strategies to overcome this problem.  The first strategy, which we follow in this article, is to replace the Nakayama functor by a procedure which is better behaved on the abelian category.  In order to  show this procedure works in the way one would hope, we are led to consider the postinjective algebra.  People normally ignore this algebra for good reason -- it is isomorphic to the preprojective algebra -- but here we find it useful to keep track of this isomorphism.  The second strategy is to work in a derived setting where the Nakayama functor is an equivalence.  I plan to return to this strategy in a sequel to this paper.

The two results above are stated for representation finite hereditary algebras, but in fact they go through just as well for representation finite $d$-hereditary algebras \cite{io-napr,hio}.  In this paper we work everything out carefully in the hereditary case in a way that generalizes immediately to the higher setting.  This is both to simplify the exposition and to demonstrate the idea that, at least in some cases, we have a good understanding of a property of hereditary algebras precisely when our explanations also work for $d$-hereditary algebras.

\emph{Acknowledgements:} 
This work was supported by the Engineering and Physical Sciences Research Council [grant number EP/G007497/1].  Thanks to Bill Crawley-Boevey and Robert Marsh for encouragement and helpful conversations.  Part of this research was conducted during a visit to King's College London.  Thanks to Konstanze Rietsch and the Mathematics School for their hospitality.

\section{Background}

We fix some notation.
If $f:L\to M$ and $g:M\to N$ are maps in a category, we denote their composition $L\to N$ by $g\circ f$, or simply 
$gf$.

Let $\F$ be a field and let $A$ be an $\F$-algebra.  We denote the category of finitely generated left $A$-modules by $A\mMod$ and write the hom space from $M$ to $N$ as $\Hom_\Lambda(M,N)$.  Let $A^\op$ denote the opposite algebra of $A$.  Then we can identify $A^\op\mMod$ with the category $\Modm A$ of finitely generated right $A$-modules.

Given $M\in A\mMod$, we denote by $M\add$ the full subcategory of objects isomorphic to direct sums of direct summands of $M$.  In particular, $A\add$ is the category of projective $A$-modules.  

\subsection{Frobenius algebras}

Let $(-)^*=\Hom_\F(-,\F)$ denote the usual duality of vector spaces.
Then $(-)^*$ interchanges left and right $A$-modules, and interchanges projective and injective modules.  In particular, the regular left (or right) module $A$ induces a right (or left) module structure on the dual $A^*$ of $A$, and $A^*\add$ is the category of injective $A$-modules.

\begin{definition}
A \emph{Frobenius algebra} is an $\F$-algebra $A$ together with an isomorphism of left modules $\varphi:A\arr\sim A^*$.
\end{definition}
Note that Frobenius algebras are necessarily finite-dimensional, and the definition is left/right symmetric because $\varphi^*$ is an isomorphism of right modules.  We are often sloppy and say ``$A$ is Frobenius'' instead of ``there exists an isomorphism $\varphi$ making $(A,\varphi)$ a Frobenius algebra''.

The Frobenius property is not invariant under Morita equivalence, but it does imply that our algebra is self-injective, which is a Morita-invariant property.  If the regular left $A$-module decomposes as $A=\bigoplus_{i=1}^nP_i^{\oplus d}$ with the modules $P_i$ pairwise nonisomorphic, so each projective appears with the same multiplicity, then $A$ is Frobenius if and only if it is self-injective.

Given a left $B$-module $M$ and an algebra map $\sigma:A\to B$, $M$ obtains the structure of a left $A$-module via $\sigma$.  If $\sigma:A\arr\sim A$ is an automorphism then the new module is called a \emph{twist} and is denoted ${}_\sigma M$.  The same process works for right modules.

The isomorphism $A\arr\sim A^*$ which comes with a Frobenius algebra can be upgraded to a bimodule isomorphism as long as we twist the action of $A$ or $A^*$ on one side.  There is a choice here: on which side of which bimodule should we twist?  Up to taking inverses, it doesn't matter.  In this article, the following definition is convenient:
\begin{definition}
Let $(A,\varphi)$ be a Frobenius algebra let $\sigma:A\to A$ be an algebra automorphism.  If $\varphi:A\arr\sim (A^*)_\sigma$ is an isomorphism of $A\da A$-bimodules then we call $\sigma$ a \emph{Nakayama automorphism} of $A$.
\end{definition}
Nakayama automorphisms are unique up to inner automorphisms of $A$.  We are often sloppy and talk about ``the'' Nakayama automorphism of an algebra.  If $A$ is Frobenius, we say that $\varphi$ and $\sigma$ constitute the \emph{Frobenius structure} of $A$.

Our aim is to describe the Frobenius structure of the preprojective algebra of an arbitrary representation finite hereditary algebra.

\subsection{The Nakayama functor and Serre duality}

Let $\Lambda$ be a finite-dimensional $\F$-algebra and let $(-)^\vee=\Hom_\Lambda(-,\Lambda)$ be duality with respect to the identity bimodule $\Lambda$.  This is a contravariant functor which interchanges left and right $\Lambda$-modules.  In general it is not an equivalence, but if we restrict to the categories of left and right projective modules we do get an equivalence $\Lambda\add\arr\sim \Lambda^\op\add$.

The composition $(-)^{\vee *}$ of the two dualities is a covariant functor, called the Nakayama functor.  We denote it $\nu:\Lambda\mMod\to \Lambda\mMod$.  Again, it is not an equivalence in general, but it does restrict to an equivalence $\Lambda\add\arr\sim \Lambda^*\add$.

The following result is well-known, and is mentioned in Section I.4.6 of Happel's book \cite{hap}.
\begin{theorem}[Serre duality]\label{thm:serre}
There is an isomorphism of vector spaces
\[ \Hom_\Lambda(P,M)\arr{\sim}\Hom_\Lambda(M,\nu P)^* \]
natural in both $P\in \Lambda\add$ and $M\in\Lambda\mMod$.
\end{theorem}
One can prove this as follows:
\[ \Hom_\Lambda(P,M) \stackrel\sim\longleftarrow P^\vee\otimes_\Lambda M\larr\sim (P^\vee\otimes_\Lambda M)^{**} \larr\sim \Hom_\Lambda(M,\nu P)^*. \]

\subsection{Auslander-Reiten theory for hereditary algebras}\label{ss:ar-hered}

For an arbitrary finite-dimensional $\F$-algebra $\Lambda$, we have the the Auslander-Reiten translations
$$\xymatrix {
\Lambda\stmod\ar@/^/[rr]^{\tau} &&
\Lambda\costmod\ar@/^/[ll]^{\tau^{-}}
}.$$
between the stable and costable module categories.
We refer to the book of Auslander, Reiten, and Smal\o{ }\cite{ars} for details.

Let $X$ be the direct sum of (a representative of each isoclass of) all $\Lambda$-modules.  In general this will be an infinite direct sum, but later we will only consider the case where $X$ has finitely many direct summands.  Then $(X/\Lambda)\add$ is the full subcategory of $\Lambda\mMod$ consisting of modules without a nonzero projective direct summand, and $(X/\Lambda^*)\add$ is the full subcategory consisting of modules without a nonzero injective direct summand.

Now let $\Lambda$ be hereditary, i.e., every submodule of a projective module is itself projective.  
Then the inclusion of $(X/\Lambda)\add$ into $X\add$ induces an equivalence of categories
\[ (X/\Lambda)\add \arr\sim \Lambda\stmod \]
and similarly we have an equivalence $(X/\Lambda^*)\add \cong \Lambda\costmod$.  Therefore we obtain the following.
\begin{proposition}\label{prop:AR-quasi-inv}
If $\Lambda$ is hereditary then $\tau$ and $\tau^{-}$ induce quasi-inverse equivalences of categories
$$\xymatrix {
(X/\Lambda)\add\ar@/^/[rr]^{\tau} &&
(X/\Lambda^*)\add\ar@/^/[ll]^{\tau^{-}}
}.$$
\end{proposition}
If we set $\tau\Lambda=0$ and $\tau^-\Lambda^*=0$ then both $\tau$ and $\tau^-$ extend to endofunctors of $\Lambda\mMod$.  There is a useful alternative description of these endofunctors which follows from the fact that, for hereditary algebras, the Auslander-Bridger transpose is given by $\Ext_\Lambda^1(-,\Lambda)$ (see, for example, \cite[Proposition 5.1.1]{ah}).
\begin{proposition}\label{prop:AR-adjoint}
If $\Lambda$ is hereditary then we have natural isomorphisms of functors 
\[ \tau^- \cong E\otimes_\Lambda-:\Lambda\mMod\to\Lambda\mMod \]
and
\[ \tau \cong \Hom_\Lambda(E,-):\Lambda\mMod\to\Lambda\mMod \]
where $E$ denotes the $\Lambda\da\Lambda$-bimodule $\Ext^1_\Lambda(\Lambda^*,\Lambda)$.  Therefore $\tau^-$ is left adjoint to $\tau$.
\end{proposition}
Modules of the form $\tau^{-r}P$, for $r\geq0$ and $P\in\Lambda\add$, are called \emph{preprojective}.  Dually, modules of the form $\tau^{r}I$, for $r\geq0$ and $I\in\Lambda^*\add$, are called \emph{postinjective} (or, confusingly, \emph{preinjective}).

\subsection{Preprojective algebras}

The following definition was given by Baer, Geigle, and Lenzing \cite{bgl}.
\begin{definition}\label{def:preproj}
Given a hereditary algebra $\Lambda$, its \emph{preprojective algebra} is
$$\Pi=\bigoplus_{r\geq0}\Hom_\Lambda(\Lambda,\tau^{-r}\Lambda)$$
with the multiplication of $f:\Lambda\to\tau^{-r}\Lambda$ and $g:\Lambda\to\tau^{-s}\Lambda$ defined as
\[ gf = \tau^{-r}g\circ f: \Lambda\to\tau^{-r-s}\Lambda. \]
\end{definition}

By construction, on restriction to a left $\Lambda$-module, $\Pi$ is isomorphic to the direct sum $\bigoplus_{r\geq0}\tau^{-r}\Lambda$ of all preprojective modules.  Ringel \cite{rin} and Crawley-Boevey \cite{cb-quiv} both showed that, up to sign, this matches the earlier generators and relations definition of the preprojective algebra of a quiver due to Gelfand and Ponomarev \cite{gp}.

Note that $\Pi$ contains $\Hom_\Lambda(\Lambda,\Lambda)\cong\Lambda^\op$ as a subalgebra.

\subsection{A criterion for finite representation type}\label{ss:finitetype}

We now consider finite-dimensional hereditary algebras $\Lambda$ of finite representation type, i.e., with only finitely many isomorphism classes of indecomposable modules.  These are the hereditary algebras for which $X\in\Lambda\mMod$, where $X$ is as in Section \ref{ss:ar-hered}. 

We will make use of the following Auslander-Reiten theoretic characterization of hereditary algebras of finite representation type due to Auslander and Platzeck \cite[Theorem 1.7]{ap} (or see \cite[Proposition VIII.1.13]{ars}) and also, in the case of quivers, to Gabriel \cite[Proposition 6.4]{gab-ar}.
\begin{theorem}\label{thm:rf-condition}
A hereditary algebra $\Lambda$ is of finite representation type if and only if every injective module is preprojective.
\end{theorem}

Let $S$ index the isoclasses of indecomposable projective left $\Lambda$-modules, so $A=\bigoplus_{i\in S}P_i^{\oplus d_i}$.  Then we have indecomposable projective right modules ${}_iP=P_i^\vee$.  Dually, we obtain indecomposable injective modules $I_i={}_iP^*$ and ${}_iI=P_i^*$.

Theorem \ref{thm:rf-condition} implies that if $\Lambda$ is of finite representation type then
there exists a permutation $\rho:S\arr\sim S$ and a function $\ell:S\to \Z_{\geq0}$ 
such that, for every $i\in S$, 
$$I_{i} \cong \tau^{-\ell(i)}P_{\rho(i)}.$$

\section{The Frobenius structure}
\subsection{The Frobenius isomorphism}

The following result was folklore before a careful proof was written down for $\Lambda$ the path algebra of a quiver by Brenner, Butler, and King \cite[Theorem 4.8]{bbk}.  
\begin{theorem}\label{thm:pi-is-frob}
If $\Lambda$ is basic and representation finite then $\Pi$ is Frobenius.
\end{theorem}
\begin{proof}
As $\Lambda=\bigoplus_{i\in S}P_i$ is basic we can decompose $\Pi$ as a vector space in the following way:
\[ \Pi = \bigoplus_{i,j\in S}\bigoplus_{r=0}^{\ell(j)}\Hom_\Lambda(P_i,\tau^{-r}P_j). \]

For each pair $i,j\in S$ and $0\leq r\leq \ell(j)$, consider the chain of isomorphisms
\[ \varphi_{ij}^r:
\Hom_\Lambda(P_i,\tau^{-r}P_j) \larr\sim
\Hom_\Lambda(\tau^{-r}P_j,I_i)^* \larr\sim
\Hom_\Lambda(P_j,\tau^rI_i)^* \larr\sim
\Hom_\Lambda(P_j,\tau^{r-\ell(i)}P_{\rho(i)})^*
\]
The first comes from Serrre duality, the second from the adjunction between $\tau^-$ and $\tau$, and the third from the fact that every injective is preprojective.  So, taking the direct sum over $i$, $j$, and $r$, we have found an isomorphism of vector spaces $\Pi\cong\Pi^*$.  It remains to show that this isomorphism preserves the left $\Pi$-module structures.

Given $f:P_i\to\tau^{-r}P_j$ and $g:P_j\to\tau^{-s}P_k$, we want to show that $g\varphi_{ij}^r(f)=\varphi_{ik}^{r+s}(gf)$, so 
we need to show that the following diagram commutes:
$$\xymatrix{
\Hom_\Lambda(P_i,\tau^{-r}P_j)\ar[r]^(0.4)\sim\ar[d] &\Hom_\Lambda(P_j,\tau^{r-\ell(i)}P_{\rho(i)})^*\ar[d]\\
\Hom_\Lambda(P_i,\tau^{-r-s}P_k)\ar[r]^(0.4)\sim &\Hom_\Lambda(P_k,\tau^{r+s-\ell(i)}P_{\rho(i)})^*
}
$$
First, we note that the square
$$\xymatrix{
\Hom_\Lambda(P_i,\tau^{-r}P_j)\ar[r]^\sim\ar[d] &\Hom_\Lambda(\tau^{-r}P_j,I_i)^*\ar[d]\\
\Hom_\Lambda(P_i,\tau^{-r-s}P_k)\ar[r]^\sim &\Hom_\Lambda(\tau^{-r-s}P_k,I_i)^*
}
$$
commutes by the naturality of Serre duality, so then we can check that the rest of the square commutes before dualizing the hom spaces, i.e., we only need to check that the following diagram commutes:
\[
\xymatrix @C=15pt{
\Hom_\Lambda(\tau^{-r}P_j,I_i)\ar[r]^\sim &\Hom_\Lambda(P_j,\tau^rI_i)\ar[r]^(0.4)\sim &\Hom_\Lambda(P_j,\tau^{r-\ell(i)}P_{\rho(i)})\\
\Hom_\Lambda(\tau^{-r-s}P_k,I_i)\ar[r]^\sim\ar[u] &\Hom_\Lambda(P_k,\tau^{r+s}I_i)\ar[r]^(0.4)\sim &\Hom_\Lambda(P_k,\tau^{r+s-\ell(i)}P_{\rho(i)})\ar[u]
}
\]
The rightmost vertical map is given by first applying $\tau^{-s}$ and then precomposing with $g$.  By Proposition \ref{prop:AR-quasi-inv}, $\tau^{-s}\tau^{r+s-\ell(i)}P_{\rho(i)}\cong\tau^{r-\ell(i)}P_{\rho(i)}$.  Thus the rightmost vertical map is the composition of the following adjunction and precomposition:
\[  \Hom_\Lambda(P_k,\tau^{r+s-\ell(i)}P_{\rho(i)}) \larr\sim 
\Hom_\Lambda(\tau^{-s}P_k,\tau^{r-\ell(\rho(i))}P_{\rho(i)}) \longrightarrow \Hom_\Lambda(P_j,\tau^{r-\ell(\rho(i))}P_{\rho(i)}). \]

We break our diagram into parts, and will show that each part commutes.
\[
\xymatrix { 
\Hom_\Lambda(\tau^{-r}P_j,I_i)\ar[r]^-\sim 
&\Hom_\Lambda(P_j,\tau^rI_i)\ar[r]^-\sim 
&\Hom_\Lambda(P_j,\tau^{r-\ell(i)}P_{\rho(i)})\\
\Hom_\Lambda(\tau^{-r-s}P_k,I_i)\ar[u]\ar[r]^-\sim
&\Hom_\Lambda(\tau^{-s}P_k,\tau^{r}I_i)\ar[u]\ar[r]^-\sim\ar[d]^-\sim
&\Hom_\Lambda(\tau^{-s}P_k,\tau^{r-\ell(i)}P_{\rho(i)})\ar[u]\ar[d]^-\sim\\
\Hom_\Lambda(\tau^{-r-s}P_k,I_i)\ar[r]^-\sim\ar@{=}[u]
&\Hom_\Lambda(P_k,\tau^{r+s}I_i)\ar[r]^-\sim
&\Hom_\Lambda(P_k,\tau^{r+s-\ell(i)}P_{\rho(i)})
}
\]
The top left and bottom right squares commute by naturality of adjunctions.  The bottom left square commutes by definition.  The top right square commutes by bifunctoriality of $\Hom_\Lambda(-,-)$.

So we have shown that our isomorphism respects the left $\Pi$-module structure, and we are done.
\end{proof}

\subsection{The graded structure}

The preprojective algebra has a natural structure of a nonnegatively graded algebra: maps $\Lambda\to\tau^{-r}\Lambda$ are in degree $r$.  Therefore the regular left and right modules also have a natural graded structure.

Let $M=\bigoplus_{r\in\Z}M_r$ be a graded (left or right) module.  We extend $(-)^*$ to a duality on graded modules by $(M^*)_r=(M_{-r})^*$.  We shift gradings by $(M\grsh{n})_r=M_{r+n}$.  With these conventions, $(M\grsh{n})^*=M^*\grsh{-n}$, and homogeneous maps $f:M\to N$ of ``degree $k$'', i.e., such that $f(M_r)\subseteq N_{r+k}$, correspond to degree $0$ maps $f:M\to N\grsh{k}$.

Our Frobenius isomorphism of Theorem \ref{thm:pi-is-frob} can be upgraded to a graded isomorphism in the following straightforward manner.  Let $\varepsilon_i:P_i\arr= P_i$ denote the identity map on $P_i$.  
Then $\varepsilon_i$ is an idempotent in $\Pi$.

Recall our function $\ell:S\to\Z_{\geq0}$ from Section \ref{ss:finitetype}.
\begin{proposition}\label{prop:graded-frob}
For all $i\in S$, we have an isomorphism $\Pi \varepsilon_i\cong (\varepsilon_{\rho(i)} \Pi)^*\grsh{-\ell(i)}$ of graded left $\Pi$-modules.
\end{proposition}
\begin{proof}
First, note that
\[\Pi \varepsilon_i = \bigoplus_{j\in S}\bigoplus_{r=0}^{\ell(j)}\Hom_\Lambda(P_i,\tau^{-r}P_j)\]
and
\[ \varepsilon_k \Pi = \bigoplus_{j\in S}\bigoplus_{r=0}^{\ell(k)}\Hom_\Lambda(P_j,\tau^{-r}P_k).\]
Recall our maps
\[ \varphi_{ij}^r:
\Hom_\Lambda(P_i,\tau^{-r}P_j) \larr\sim
\Hom_\Lambda(P_j,\tau^{r-\ell(i)}P_{\rho(i)})^*.
\]
So elements in $(\Pi \varepsilon_i)_r$ are sent to elements in the dual of $(\varepsilon_{\rho(i)}\Pi)_{-r+\ell(i)}$, i.e., $\varphi_{ij}^r$ sends degree $r$ elements of $\Pi \varepsilon_i$ to degree $r-\ell(i)$ elements of $(\varepsilon_{\rho(i)}\Pi)^*$.  So $\varphi_{ij}^r$  is a homogeneous map of degree $-\ell(i)$ and thus
\[ \varphi_i=\bigoplus_{j,r}\varphi_{ij}^r \]
is a graded map $\Pi \varepsilon_i\to (\varepsilon_{\rho(i)} \Pi)^*\grsh{-\ell(i)}$.
\end{proof}
We say that $\Pi$ is a graded Frobenius algebra with Gorenstein function $\ell$.  
Informally, we write 
\[{}_\Pi\Pi\cong {}_\Pi\Pi^*\grsh{\ell}. \]

\subsection{Postinjective algebras}

There is an obvious definition dual to the preprojective algebra:
\begin{definition}\label{def:postinj}
Given a hereditary algebra $\Lambda$, its \emph{postinjective algebra} is
$$ \iP =\bigoplus_{r\geq0}\Hom_\Lambda(\tau^{r}\Lambda^*,\Lambda^*)$$
with the multiplication of $f:\tau^{r}\Lambda^*\to\Lambda^*$ and $g:\tau^{s}\Lambda^*\to\Lambda^*$ defined as
\[ gf = g\circ \tau^sf: \tau^{r+s}\Lambda\to\Lambda. \]
$ \iP$ is nonnegatively graded with maps $\tau^{r}\Lambda^*\to\Lambda^*$ in degree $r$.
\end{definition}
This gives us nothing new, but it will be useful to see precisely how this gives us nothing new.  We go from $\Pi$ to $\iP$ using Serre duality and our adjunctions, as follows:
\[
\Delta_{ij}^r:\Hom_\Lambda(P_i,\tau^{-r}P_j)\arr\sim\Hom_\Lambda(\tau^{-r}P_j,I_i)^*\arr\sim\Hom_\Lambda(P_j, \tau^{r}I_i)^*\arr\sim\Hom_\Lambda(\tau^{r}I_i,I_j)^{**}\arr\sim\Hom_\Lambda(\tau^{r}I_i,I_j)
\]
\begin{proposition}\label{prop:PiiP}
The map $\Delta:\Pi\to\iP$ sending $f:P_i\to\tau^{-r}P_j$ to $\Delta(f):\tau^rI_i\to I_j$ is an isomorphism of graded algebras.
\end{proposition}
\begin{proof}
We only need to check that $\Delta$ is an algebra homomorphism, i.e., that the following diagram (with $\Hom$s omitted) commutes:
\[
\xymatrix { 
(P_i,\tau^{-r}P_j)\otimes (P_j,\tau^{-s}P_k)\ar[r]\ar[d]^\sim & (P_i,\tau^{-r-s}P_k)\ar[d]^\sim \\
(\tau^{r}I_i,I_j)\otimes (\tau^{s}I_j,I_k)\ar[r] & (\tau^{r+s}I_i,I_k)
}
\]
We use the definitions of the multiplication maps and break the diagram up as follows:
\[
\xymatrix { 
(P_i,\tau^{-r}P_j)\otimes (\tau^{-r}P_j,\tau^{-r-s}P_k)\ar[r]\ar[d] & (P_i,\tau^{-r-s}P_k)\ar[d] \\
(\tau^{-r}P_j,I_i)^*\otimes (\tau^{-r}P_j,\tau^{-r-s}P_k)\ar[r]\ar[d] & (\tau^{-r-s}P_k,I_i)^*\ar[d] \\
(P_j,\tau^{r}I_i)^*\otimes (P_j,\tau^{-s}P_k)\ar[r]\ar[d] & (\tau^{-s}P_k,\tau^{r}I_i)^*\ar@{=}[d] \\
(\tau^{r}I_i,I_j)^{**}\otimes (\tau^{-s}P_k,I_j)^*\ar[r]\ar[d] & (\tau^{-s}P_k,\tau^{r}I_i)^*\ar[d] \\
(\tau^{r+s}I_i,\tau^sI_j)^{**}\otimes (P_k,\tau^s I_j)^*\ar[r]\ar[d] & (P_k,\tau^{r+s}I_i)^*\ar[d] \\
(\tau^{r+s}I_i,\tau^sI_j)^{**}\otimes (\tau^{s}I_j,I_k)^{**}\ar[r] & (\tau^{r+s}I_i,I_k)^{**}
}
\]
The constituent squares commute because of the naturality of Serre duality and the naturality of the adjunction, alternatively.  
\end{proof}
We obtain two corollaries.  The first can be seen as an enhanced naturality property of Serre duality, and is a useful tool for showing that diagrams commute.
\begin{corollary}\label{cor:nat+}
Fix $f:P_i\to\tau^{-r}P_j$.  Then the following diagram commutes:
\[
\xymatrix{
\Hom_\Lambda(P_j,\tau^{-s} P_k)\ar[r]^\sim\ar[d]^{\tau^{-r}} & \Hom_\Lambda(\tau^{-s} P_k,I_j)^*\ar[d]^{(1,\Delta(f))^*} \\
\Hom_\Lambda(\tau^{-r}P_j,\tau^{-r-s} P_k)\ar[d]^{(f,1)} & \Hom_\Lambda(\tau^{-s} P_k,\tau^r I_i)^*\ar[d]_\sim^{\operatorname{adj}} \\
\Hom_\Lambda(P_i,\tau^{-r-s} P_k)\ar[r]^\sim & \Hom_\Lambda(\tau^{-r-s} P_k,I_i)^*
}
\]
\end{corollary}
\begin{proof}
Having fixed $f$, this is a subdiagram of the large diagram in the previous proof.
\end{proof}
The second relates the (``left'') preprojective algebra to the ``right'' preprojective algebra.
By Proposition \ref{prop:AR-adjoint}, $\tau^-:\Lambda\mMod\to\Lambda\mMod$ is given by tensoring on the left with the bimodule $E=\Ext^1_\Lambda(\Lambda^*,\Lambda)\cong\Ext^1_{\Lambda\otimes_\F\Lambda^\op}(\Lambda,\Lambda\otimes_\F\Lambda)$, which is left-right symmetric.  So $\tau^-:\Lambda^\op\mMod\to\Lambda^\op\mMod$ is given by tensoring on the right with $E$, and thus we can write $\tau^-\Lambda$ unambiguously whether we are dealing with left of right modules.
\begin{corollary}\label{cor:lrprep}
 We have an algebra anti-isomorphism
\begin{align*}
R: \bigoplus_{r\geq0}\Hom_\Lambda(\Lambda,\tau^{-r}\Lambda)\arr\sim &\bigoplus_{r\geq0}\Hom_{\Lambda^\op}(\Lambda,\tau^{-r}\Lambda)\\
 f\mapsto &\ell_{f(1)}
\end{align*}
which respects the grading.  It sends a map $f:P_i\to\tau^{-r}P_j$ to a map $R(f):{_jP}\to\tau^{-r}{_iP}$.
\end{corollary}
\begin{proof}
Take the $\F$-linear dual of $\iP$.
\end{proof}
We therefore have the following diagram of graded algebras:
\[
\xymatrix{
\text{left preprojective algebra} \ar[r]^\sim &
\text{left postinjective algebra} \ar[d]^{\text{anti-isom}} \\
\text{right postinjective algebra} \ar[u]^{\text{anti-isom}} &
 \text{right preprojective algebra} \ar[l]_\sim
}
\]
\begin{remark}
If one just wanted to relate the left and right preprojective algebras, an alternative approach would be to use the fact that both functors $\Hom_\Lambda(\Lambda,-)$ and $\Hom_{\Lambda^\op}(\Lambda,-)$ are naturally isomorphic to the identity functor on the category of $\Lambda\otimes_\F\Lambda^\op$-modules.  Explicitly, we have isomorphisms \[\Hom_\Lambda(\Lambda,M)\arr\sim \Hom_{\Lambda^\op}(\Lambda,M)\]
given by $f\mapsto [a\mapsto f(1)a]$.  Then one can check on elements that this gives an algebra anti-isomorphism.
\end{remark}

\subsection{The Nakayama automorphism}\label{ss:nak}

We will now describe the Nakayama automorphism of $\Pi$, denoted $\sigma$.
Given a map
$$f:P_i\to \tau^{-r}P_j$$
of left $\Lambda$-modules we first apply the algebra isomorphism $\Delta$ of Proposition \ref{prop:PiiP}, and pre- and post-compose with the isomorphisms which exhibit our injectives as preprojective modules, to obtain a map
\[ \tau^{r-\ell(i)} P_{\rho(i)}\to \tau^{-\ell(j)}P_{\rho(j)} \]
between preprojectives.  
Then we use the adjunction between $\tau^-$ and $\tau$ to obtain:
\[ \sigma(f):{P_{\rho(i)}} \to \tau^{\ell(i)-\ell(j)-r}{P_{\rho(j)}}. \]

\begin{theorem}\label{thm:nak}
$\sigma$ is a Nakayama automorphism of $\Pi$, i.e., we have an isomorphism
$$\Pi\cong(\Pi^*)_{\sigma}$$
of $\Pi\da\Pi$-bimodules.
\end{theorem}
\begin{proof}
Choose some
$$f:P_i\to \tau^{-r}P_j$$
and
$$g:P_j\to \tau^{-s}P_k$$
in $\Pi$.  
We apply the chain of isomorphisms of hom spaces from Theorem \ref{thm:pi-is-frob} to $g$:
$$
\Hom_\Lambda(P_j,\tau^{-s}P_k)\cong\Hom_\Lambda(\tau^{-s}P_k,I_j)^*\cong\Hom_\Lambda(P_k,\tau^sI_j)^*\cong\Hom_\Lambda(P_k,\tau^{s-\ell(j)}P_{\rho(j)})^*$$
We need to show that the following diagram commutes:
$$\xymatrix{
\Hom_\Lambda(P_j,\tau^{-s}P_k)\ar[r]^(0.4)\sim\ar[d] &\Hom_\Lambda(P_k,\tau^{s-\ell(j)}P_{\rho(j)})^*\ar[d]\\
\Hom_\Lambda(P_i,\tau^{-r-s}P_k)\ar[r]^(0.4)\sim &\Hom_\Lambda(P_k,\tau^{r+s-\ell(i)}P_{\rho(i)})^*
}
$$
The left vertical map is the right action of $f$, so we apply $\tau^{-r}$ to $g$ and then precompose with $f$, and the right vertical map is given by post-composition with $\tau^{r+s-\ell(i)}\sigma(f)$ before applying the functional.

By Corollary \ref{cor:nat+} we have that the following diagram commutes:
$$
\xymatrix{
\Hom_\Lambda(P_j,\tau^{-s} P_k)\ar[r]^\sim\ar[d] & \Hom_\Lambda(\tau^{-s} P_k,I_j)^*\ar[d] \\
\Hom_\Lambda(\tau^{-r}P_j,\tau^{-r-s} P_k)\ar[d] & \Hom_\Lambda(\tau^{-s} P_k,\tau^{r} I_i)^*\ar[d]^\sim \\
\Hom_\Lambda(P_i,\tau^{-r-s} P_k)\ar[r]^\sim & \Hom_\Lambda(\tau^{-r-s} P_k,I_i)^*
}
$$
and so we can remove duals from the rest of the diagram, and it remains to check that the following diagram commutes:
$$
\xymatrix{
\Hom_\Lambda(\tau^{-s} P_k,I_j)\ar[r]^\sim &\Hom_\Lambda(P_k,\tau^sI_j)\ar[r]^(0.4)\sim &\Hom_\Lambda(P_k,\tau^{s-\ell(j)}P_{\rho(j)})\\
\Hom_\Lambda(\tau^{-s} P_k,\tau^{r} I_i)\ar[u] &&\\
\Hom_\Lambda(\tau^{-r-s} P_k,I_i)\ar[u]^\sim\ar[r]^\sim &\Hom_\Lambda(P_k,\tau^{r+s}I_i)\ar[r]^(0.45)\sim &\Hom_\Lambda(P_k,\tau^{r+s-\ell(i)}P_{\rho(i)})\ar[uu]
}
$$
The left vertical map is defined by first using the adjunction and then postcomposing with ${\Delta(f)}$, and the
right vertical map is post-composition with $\tau^{r+s-\ell(i)}\sigma(f)$.  By definition of $\sigma$ and the horizontal maps, the following diagram commutes:
$$
\xymatrix{
\Hom_\Lambda(P_k,\tau^sI_j)\ar[r]^(0.42)\sim &\Hom_\Lambda(P_k,\tau^{s-\ell(j)}P_{\rho(j)})\\
\Hom_\Lambda(P_k,\tau^{r+s}I_i)\ar[r]^(0.42)\sim\ar[u] &\Hom_\Lambda(P_k,\tau^{r+s-\ell(i)}P_{\rho(i)})\ar[u]
}
$$
where the left vertical map is post-composition with $\tau^s \Delta(f):\tau^{r+s}I_i\to \tau^sI_j$.  So it remains to check that the square
$$
\xymatrix{
\Hom_\Lambda(\tau^{-s} P_k,I_j)\ar[r]^\sim &\Hom_\Lambda(P_k,\tau^sI_j)\\
\Hom_\Lambda(\tau^{-s} P_k,\tau^{r} I_i)\ar[u] &\\
\Hom_\Lambda(\tau^{-r-s} P_k,I_i)\ar[u]^\sim\ar[r]^\sim &\Hom_\Lambda(P_k,\tau^{r+s}I_i)\ar[uu]
}
$$
commutes.  The bottom left vertical map and the bottom horizontal maps are both just adjunction maps, so the triangle in the following square commutes
$$
\xymatrix{
\Hom_\Lambda(\tau^{-s} P_k,I_j)\ar[r]^\sim &\Hom_\Lambda(P_k,\tau^sI_j)\\
\Hom_\Lambda(\tau^{-s} P_k,\tau^{r} I_i)\ar[u]\ar[dr]^\sim &\\
\Hom_\Lambda(\tau^{-r-s} P_k,I_i)\ar[u]^\sim\ar[r]^\sim &\Hom_\Lambda(P_k,\tau^{r+s}I_i)\ar[uu]
}
$$
and the remaining square commutes by the naturality of the adjunction.
\end{proof}

For $\Lambda=\F Q$ the path algebra of an ADE Dynkin quiver, an explicit formula for the Nakayama automorphism was given by Brenner, Butler, and King.  In this case, there is at most one arrow of $Q$ with a given source and target, so we can write $y_{ij}$ for an arrow $\alpha:i\to j$ of $Q$.
Recall that $\Pi$ is generated by the arrows $y_{ij}=\alpha:i\to j$ of $Q$ and by arrows $y_{ji}=\alpha^*:j\to i$ in the opposite direction.  The former corresponds to irreducible maps between projectives and the latter to irreducible maps from projectives to summands of $E=\tau^-\Lambda$.  Then we have the following formula \cite[{Definition 4.6 and Theorem 4.8}]{bbk}:
\begin{theorem}[The BBK formula]
For $Q$ an ADE Quiver, the Nakayama automorphism $\sigma$ of the preprojective algebra of $\Lambda=\F Q$ acts on generators by:
\[ \sigma(y_{ij})=
\begin{cases}
      y_{\rho(i)\rho(j)} & \text{if $y_{ij}\in Q$ or $y_{\rho(i)\rho(j)}\in Q$};\\
      -y_{\rho(i)\rho(j)} & \text{if $y_{ij}\notin Q$ and $y_{\rho(i)\rho(j)}\notin Q$}.
    \end{cases}
\]
\end{theorem}
\begin{remark}
It would be interesting to recover the BBK formula directly from Theorem \ref{thm:nak}.  The hope is that such a derivation would avoid the complicated calculations necessary in Section 4.4 of \cite{bbk}.  Such calculations become more difficult in higher global dimension, as illustrated by the appendix of the arXiv version of \cite{ep}.
\end{remark}

\subsection{Worked example}
To make calculations it is useful to use the isomorphism $\tau^-\cong E\otimes_\Lambda-$ from Proposition \ref{prop:AR-adjoint} and the classical fact that, for $\Lambda =\F Q$, the bimodule $E$ is generated by the dual of the arrow space of $Q$ (see \cite[Proposition 3.1]{gi} for a modern treatment).

\begin{example}
Let $Q$ be the quiver
\[  1\stackrel\alpha\longrightarrow 2\stackrel\beta\longrightarrow 3\stackrel \gamma \longleftarrow 4\]
and let $\Lambda=\F Q$ be the path algebra of $\Lambda$.  Unlike functions, we compose arrows in a path right to left, so the nontrivial length $2$ path is denoted $\beta\alpha$.  We write the projective left $\Lambda$-modules as $P_i=\Lambda e_i$ and the preprojectives as $E^r_i=E^{\otimes_\Lambda r}e_i$.  
They have bases
\begin{align*}
&P_1 = \gen { e_1}, \;\;\; &&P_2 = \gen { e_2, \alpha}, \;\;\;  &&P_3 = \gen { e_3, \beta, \gamma, \alpha\beta}, \;\;\; &&P_4 = \gen { e_4},\\
&E_1 = \gen { \alpha^*}, \;\;\; &&E_2 = \gen { \beta^*, \beta\beta^*, \gamma\beta^*}, \;\;\;  &&E_3 = \gen { \beta^*\beta,  \beta\beta^*\beta}, \;\;\; &&E_4 = \gen { \gamma^*, \beta\gamma^*,\alpha\beta\gamma^*}, \\
&E^2_1=\gen{\beta^*\alpha^*, \gamma\beta^*\alpha^*}, \;\;\; &&E^2_2=\gen{\beta^*\alpha^*\alpha } && &&
\end{align*}
where $\alpha^*\alpha=\beta\beta^*$ and $\beta^*\beta=-\gamma^*\gamma$.

The Auslander-Reiten quiver of $\Lambda\mMod$ is as follows:
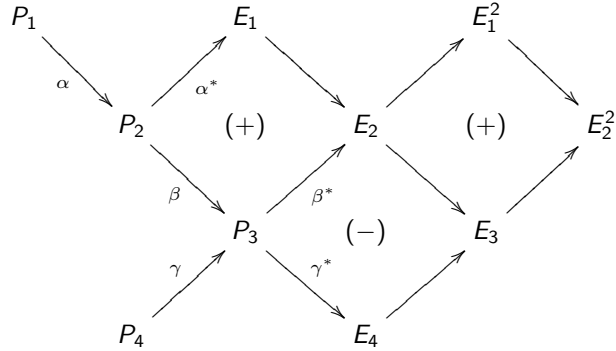
\[ \xymatrix{ 
P_1\ar[dr]_{\alpha} &&E_1\ar[dr] &&E^2_1\ar[dr] & \\
& P_2\ar[ur]_{\alpha^*}\ar[dr]_\beta &(+)&  E_2\ar[ur]\ar[dr] &(+)& E^2_2 \\
&&  P_3\ar[ur]_{\beta^*}\ar[dr]^{\gamma^*} &(-)& E_3\ar[ur] & \\
& P_4\ar[ur]^\gamma &&E_4\ar[ur] &&
}\]
The leftmost arrows are labelled by arrows and their duals: the map $P_1\to P_2$ sends $e_1$ to $\alpha$, and the map $P_2\to E_1$ sends $e_2$ to $\alpha^*$.  These can be seen as right multiplication, and then the arrows obtained by applying $\tau^-$ can be interpreted as the same multiplications: for example, the map $E_1\to E_2$ sends $\alpha^*\mapsto\alpha^*\alpha=\beta\beta^*$.
Note that not all squares commute: the anticommuting square is marked $(-)$.  

The injective modules have bases
\[ I_1=\gen{\beta^*\alpha^*,\alpha^*,e_1^*},\;\;\; I_2=\gen{\beta^*, e_2^*},\;\;\; I_3=\gen{e_3^*},\;\;\; I_4=\gen{\gamma^*,e_4^*}. \]
Our vertex permutation $\rho$ sends $i$ to $5-i$ and our length function is given by $\ell(1)=\ell(2)=2$ and $\ell(3)=\ell(4)=1$.  
We fix isomorphisms 
$I_i\cong E^{\ell(i)}_{\rho(i)}$ 
which are determined by
\[ 
I_1\ni e_1^*\mapsto \alpha\beta\gamma^*\in E_4, \;\;\;
I_2\ni e_2^*\mapsto \beta\beta^*\beta\in E_3, \;\;\;
I_3\ni e_3^*\mapsto \beta^*\alpha^*\alpha \in E^2_2, \;\;\;
I_4\ni e_4^*\mapsto \gamma\beta^*\alpha^* \in E^2_1.
\]

We start by applying the maps $\varphi_{ij}^r$ of Theorem \ref{thm:pi-is-frob} to the maps out of $P_1$, so $i=1$.  So we work with
\[ \Pi\varepsilon_1=\gen{ \; \id_{P_1}:P_1\to P_1, \;\;\; r_\alpha:P_1\to P_2, \;\;\;  r_{\alpha\beta}:P_1\to P_3, \;\;\; r_{\alpha\beta\gamma^*}:P_1\to E_4\; }. \]
Serre duality sends $\id_{P_1}$ to the dual of the map $P_1\to I_1$ sending $e_1\mapsto e_1^*$.  Using our isomorphism $I_1\cong E_4$ we get the dual of $r_{\alpha\beta\gamma^*}:P_1\to E_4$.  Similarly, the degree $0$ maps $r_\alpha:P_1\to P_2$ and $r_{\alpha\beta}$ are sent to the duals of $r_{\beta\gamma^*}:P_2\to E_4$ and $r_{\gamma^*}:P_3\to E_4$, respectively.  Serre duality sends the degree $1$ map $r_{\alpha\beta\gamma^*}$ to the dual of the map $E_4\to I_1$ sending $\alpha\beta\gamma^*\mapsto e_1^*$.  Composing with $I_1\cong E_4$, this sends $\alpha\beta\gamma^*$ to itself, so on applying $\tau$ we get $\id_{P_4}$.  This shows that $(\Pi\varepsilon_1)^*\cong \varepsilon_4\Pi\grsh{1}$, as predicted by the dual of Proposition \ref{prop:graded-frob}.

Now we calculate the Nakayama automorphism $\sigma$.  The identity map on $P_i$ is clearly sent to the identity map on $P_{\rho(i)}$ and $\sigma$ is an algebra automorphism, so it is enough to find the images of the irreducible maps $P_i\to P_j$ and $P_i\to E_j$.

We start with irreducible maps $P_i\to P_j$.
First, $\Delta$ sends $r_\alpha:P_1\to P_2$ to the map $I_1\to I_2$ sending $\alpha^*$ to $e_2^*$.  So using our fixed isomorphisms we get the map $E_4\to E_3$ sending $\beta\gamma^*$ to $\beta\beta^*\beta=-\beta\gamma^*\gamma$.  So, applying $\tau^-$, we see that $\sigma(r_\alpha)$ is the map $-r_\gamma:P_4\to P_3$ sending $e_4$ to $-\gamma$.  Second, $\Delta(r_\beta)$ sends $\beta^*\in I_2$ to $e_3^*\in I_3$, so we get $E_3\to E^2_2$ sending $\beta^*\beta$ to $\beta^*\alpha^*\alpha=\beta^*\beta\beta^*$, and applying $\tau^-$ gives $r_{\beta^*}:P_3\to E_2$.  Third, $\Delta(r_\gamma)$ sends $\gamma^*\in I_4$ to $e_3^*\in I_3$, so we get $E^2_2\to E^2_1$ sending $\beta^*\alpha^*$ to $\beta^*\alpha^*\alpha$, and applying $\tau^-$ gives $r_{\alpha}:P_1\to P_2$.

Now we determine how $\sigma$ acts on irreducible maps $P_i\to E_j$.   First, consider $r_{\alpha^*}:P_2\to E_1$.   Applying Serre duality gives the dual of a map $E_1\to I_2$ which sends $\alpha^*$ to $e_2^*$.  Under our isomorphism $I_2\cong E_3$, $e_2^*$ corresponds to $\beta\beta^*\beta=\alpha^*\alpha\beta$, so applying $\tau^-$ gives the map $r_{\alpha\beta}:P_1\to P_3$.  Using Serre duality once more gives us the map $P_3\to I_1$ sending $\alpha\beta$ to $e_1^*$, which corresponds to $\alpha\beta\gamma^*$ in $E_4$.  So $\sigma(r_{\alpha^*})=r_{\gamma^*}:P_3\to E_4$.  

Second, we take $r_{\beta^*}:P_3\to E_2$.  Applying Serre duality gives the dual of a map $E_2\to I_3$ which sends $\beta^*$ to $e_3^*$.  Under our isomorphism $I_3\cong E^2_2$, $e_3^*$ corresponds to $\beta^*\alpha^*\alpha$, so applying $\tau^-$ gives the map $r_{\alpha^*\alpha}:P_2\to E_2$.  Using Serre duality once more gives us the map $E_2\to I_2$ sending $\alpha^*\alpha=\beta\beta^*$ to $e_2^*$, which corresponds to $\beta\beta^*\beta$ in $E_3$.  Applying $\tau$ once more shows that $\sigma(r_{\beta^*})=r_{\beta}:P_2\to P_3$. 

Third, consider $r_{\gamma^*}:P_3\to E_4$.  Applying Serre duality gives the dual of a map $E_4\to I_3$ which sends $\gamma^*$ to $e_3^*$.  Under our isomorphism $I_3\cong E^2_2$, $e_3^*$ corresponds to $\beta^*\alpha^*\alpha=\beta^*\beta\beta^*=-\gamma^*\gamma\beta^*$, so applying $\tau^-$ gives the map $-r_{\gamma\beta^*}:P_4\to E_2$.  Using Serre duality once more gives us the map $E_2\to I_4$ sending $-\gamma\beta^*$ to $e_4^*$, which corresponds to $\gamma\beta^*\alpha^*$ in $E_3$.  Applying $\tau$ once more shows that $\sigma(r_{\gamma^*})=-r_{\alpha}:P_2\to P_3$. 

Below, we have two copies of the quiver of $\Pi^\op$, with $Q$ as the subquiver on the top arrows.  
In the first copy we record the sign of its image under the Nakayama automorphism as calculated above.  The second copy shows the sign appearing in the BBK formula.
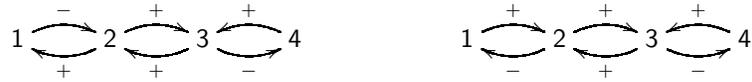
\[
\xymatrix{
1 \ar@/^/[r]^{-} & 2\ar@/^/[l]^{+}\ar@/^/[r]^{+} & 3\ar@/^/[l]^{+}\ar@/_/[r]_{-} & 4\ar@/_/[l]_{+} && 1 \ar@/^/[r]^{+} & 2\ar@/^/[l]^{-}\ar@/^/[r]^{+} & 3\ar@/^/[l]^{+}\ar@/_/[r]_{-} & 4\ar@/_/[l]_{+}
}
\]
These automorphisms differ by an inner automorphism of $\Pi$: we can conjugate by $-\varepsilon_1+\varepsilon_2+\varepsilon_3+\varepsilon_4$.  To recover the BBK automorphism directly, we should scale our isomorphism $I_1\cong E_4$ by $-1$ so that it sends $e_1^*\mapsto -\alpha\beta\gamma^*$.
\end{example}


\section{Higher preprojective algebras}\label{sec:higher}

\subsection{Generalizations}

All the constructions and proofs in this paper generalize immediately to $d$-representation finite and $(d+1)$-preprojective algebras \cite{io-napr}.  Indeed, they were originally conceived in this setting, but for the sake of clear exposition we have explained everything in the classical setting.  To modify the statements and proofs, one only needs to add a subscript $d$ to each $\tau$ and $\tau^-$ \cite{iya-higher-ar}. 

Of course, one needs to know the theorems we rely on still hold in this generality, but they do:
\[
\begin{tabular}{c|c}
Result used above & $d$-AR generalization \\
  \hline \hline
  Proposition \ref{prop:AR-quasi-inv} & Theorem 1.4.1 of \cite{iya-higher-ar} \\
  Proposition \ref{prop:AR-adjoint} & Proposition 2.10 of \cite{gi}, \\&based on Lemma 2.13 of \cite{io-stab} \\
  Definition \ref{def:preproj} & Definition 2.11 of \cite{io-stab} \\
  Theorem \ref{thm:rf-condition} & Proposition 1.3(b) of \cite{iya-ct-higher}
\end{tabular}
\]

Therefore, if $\Lambda$ is a $d$-representation finite algebra, and $\Pi$ is its $(n+1)$-preprojective algebra, we get the following results.

\begin{theorem}\label{thm:pi-is-frob2}
If $\Lambda$ is basic and $d$-representation finite then $\Pi$ is Frobenius.
\end{theorem}
\begin{proof}
As for Theorem \ref{thm:pi-is-frob}.
\end{proof}
Then, defining $\sigma$ as in Section \ref{ss:nak} but with $\tau^-_d$ and $\tau_d$, we get
\begin{theorem}\label{thm:nak2}
$\sigma$ is a Nakayama automorphism of $\Pi$, i.e., we have an isomorphism
$$\Pi\cong(\Pi^*)_{\sigma}$$
of $\Pi\da\Pi$-bimodules.
\end{theorem}
\begin{proof}
As for Theorem \ref{thm:nak}.
\end{proof}

To make calculations, we use that $\tau^-_2\cong E\otimes_\Lambda-$ where $E$ is generated by the dual of the relation space of $\Lambda$.  
\begin{example}
Let $\Lambda=\F Q/I$ where $Q$ is the quiver
$ 1\stackrel\alpha\longrightarrow 2\stackrel\beta\longrightarrow 3 $ and $I=\gen{\alpha\beta}$.  $\Lambda$ is isomorphic to the Auslander algebra of the path algebra of the quiver $1\to 2$ and so is $2$-representation finite.  The projective and $2$-preprojective left $\Lambda$-modules are
\[ P_1 = \gen { e_1}, \;\;\; P_2 = \gen { e_2, \alpha}, \;\;\;  P_3 = \gen { e_3, \beta}, \;\;\; E_1=\gen{ \beta^*\alpha^*}. \]
The $2$-Auslander-Reiten quiver of $\Lambda\mMod$ is as follows:
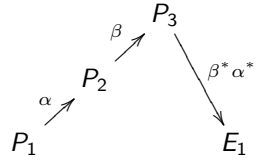
\[ \xymatrix @=10pt { 
&&  P_3\ar[ddr]^{\beta^*\alpha^*} & \\
& P_2\ar[ur]^{\beta} && \\
P_1\ar[ur]^{\alpha} &&& E_1
}\]
Note that the composition $P_1\to P_2\to P_3$ is zero.

The injective modules have bases
\[ I_1=\gen{\alpha^*,e_1^*},\;\;\; I_2=\gen{\beta^*, e_2^*},\;\;\; I_3=\gen{e_3^*}. \]
Our vertex permutation is given by $\rho(1)=2$, $\rho(2)=3$, and $\rho(3)=1$ and our length function is given by $\ell(1)=\ell(2)=0$ and $\ell(3)=1$.  
We fix isomorphisms 
$I_i\cong E^{\ell(i)}_{\rho(i)}$ 
which are determined by
\[ 
I_1\ni e_1^*\mapsto \alpha\in P_2, \;\;\;
I_2\ni e_2^*\mapsto \beta\in P_3, \;\;\;
I_3\ni e_3^*\mapsto \beta^*\alpha^*\in E_1.
\]

Using Serre duality it is easy to check that
\[ (\Pi\varepsilon_1)^*\cong \varepsilon_2\Pi\grsh{0}, \;\;\; (\Pi\varepsilon_2)^*\cong \varepsilon_3\Pi\grsh{0}, \;\;\; (\Pi\varepsilon_3)^*\cong \varepsilon_1\Pi\grsh{1} \]
as predicted by the generalization of the dual of Proposition \ref{prop:graded-frob}.

Now we calculate the Nakayama automorphism $\sigma$.  $\Delta$ sends $r_\alpha:P_1\to P_2$ to the map $I_1\to I_2$ sending $\alpha^*$ to $e_2^*$, which corresponds to sending $e_2\in P_2$ to $\beta\in P_3$, so $\sigma(r_\alpha)=r_\beta:P_2\to P_3$.  Similarly, $\sigma(r_\beta)=r_{\beta^*\alpha^*}:P_3\to E_1$.  Finally, applying Serre duality to $r_{\beta^*\alpha^*}$ gives the dual of the map $E_1\to I_3$ sending $\beta^*\alpha^*$ to $e_1$ which, using our isomorphism $I_3\cong E_1$, gives the identity map on $E_3$.  So applying $\tau_2^-$ gives the identity map on $P_1$, and applying Serre duality once more gives $P_1\to I_1$ sending $e_1$ to $e_1^*$.  Using $I_1\cong P_2$ gives the map $P_1\to P_2$ sending $e_1$ to $\alpha$, so we get $\sigma(r_{\beta^*\alpha^*})=r_\alpha:P_1\to P_2$.  Note that this agrees with the calculation of Herschend and Iyama \cite[Theorem 3.5]{hi-frac}.
\end{example}


\end{document}